\documentclass[preprint,12pt]{elsarticle}
\usepackage{amsmath}
\usepackage{amssymb}

\input{tcilatex}

\begin{document}

\begin{frontmatter}



\title{
FIXED POINT THEOREMS AND APPLICATIONS IN THEORY OF GAMES
}

\author{Monica Patriche}

\address{
University of Bucharest
Faculty of Mathematics and Computer Science
    
14 Academiei Street
   
 010014 Bucharest, 
Romania
    
monica.patriche@yahoo.com }

\begin{abstract}
We introduce the notions of weakly *-concave and weakly naturally quasi-concave 
correspondence and prove fixed point theorems and continuous selection theorems for 
these kind of correspondences. As applications in the game theory, by using a tehnique 
based on a continuous selection, we establish new existence results for the equilibrium of 
the abstract economies. The constraint correspondences are weakly naturally quasi-concave. 
We show that the equilibrium exists without continuity assumptions. 

\end{abstract}


\end{frontmatter}



\label{}





\bibliographystyle{elsarticle-num}
\bibliography{<your-bib-database>}







\section{Introduction}

\bigskip It is known that the theory of correspondences has very widely
developed and produced many applications, especially during the last few
decades. Most of these applications concern fixed point theory and game
theory. The fixed point theorems are closely connected with convexity. A
considerable number of papers devotes to correspondences on nonconvex and
noncompact domains\ (see e.g. [16], [17], [18]) or to correspondences
without convex values and continuity ([5]).

The aim of this paper is to introduce the notions of weakly *-concave and
weakly naturally quasi-concave correspondence and prove fixed point theorems
and continuous selection theorems for these kind of correspondences. We also
define the correspondences with WNQS and e-WNQS property.

The applications concern the equilibrium theory: we establish new existence
results for the equilibrium of the abstract economies. The constraint
correspondences are weakly concave-like or have WNQS, respectively e-WNQS
property.

For the reader's convenience, we review the main results in the equilibrium
theory, emphasizing that most authors have studied the existence of
equilibrium for abstract economies with preferences represented as
correspondences which have continuity properties. We mention here the
results obtained by W. Shafer and H. Sonnenschein [14], which concern
economies with finite dimensional commodity space and preference
correspondences having an open graph. N. C. Yannelis and N. D. Prahbakar
[21] used selection theorems and fixed-point theorems for correspondences
with open lower sections defined on infinite dimensional strategy spaces.
Some authors developed the theory of continuous selections of
correspondences and gave numerous applications in game theory. Michael's
selection theorem [11] is well-known and basic in many applications. In
[3,4], F. Browder firstly used a continuous selection theorem to prove
Fan-Browder fixed point theorem. Later, N. C. Yannelis and N. D. Prabhakar
[21], H. Ben-El-Mechaiekh [1], X. Ding, W. Kim and K.Tan [6], C.Horvath [9],
T. Husain and E. Taradfar [10], S.Park [12],[13], X. Wu [19], X. Wu and S.
Shen [20], Z. Yu and L. Lin [22] and many others established several
continuous selection theorems with applications.

In this paper, we show that an equilibrium for an abstract economy exists
without continuity assumptions. By using a tehnique based on a continuous
selection, we prove the new equilibrium existence theorem for an abstract
economy.

The paper is organized in the following way: Section 2 contains
preliminaries and notations. The fixed point and the selection theorem are
presented in Section 3. The equilibrium theorems are stated in Section 4.

\section{\textbf{PRELIMINARIES AND NOTATIONS\protect\smallskip \protect%
\medskip }}

Throughout this paper, we shall use the following notations and definitions:

Let $A$ be a subset of a topological space $X$.

1. 2$^{A}$ denotes the family of all subsets of $A$.

2. cl $A$ denotes the closure of $A$ in $X$.

3. If $A$ is a subset of a vector space, co$A$ denotes the convex hull of $A$%
.

4. If $F$, $T:$ $A\rightarrow 2^{X}$ are correspondences, then co$T$, cl $T$%
, $T\cap F$ $:$ $A\rightarrow 2^{X}$ are correspondences defined by $($co$%
T)(x)=$co$T(x)$, $($cl$T)(x)=$cl$T(x)$ and $(T\cap F)(x)=T(x)\cap F(x)$ for
each $x\in A$, respectively.

5. The graph of $T:X\rightarrow 2^{Y}$ is the set Gr$(T)=\{(x,y)\in X\times
Y\mid y\in T(x)\}$

6. The correspondence $\overline{T}$ is defined by $\overline{T}(x)=\{y\in
Y:(x,y)\in $cl$_{X\times Y}$Gr$T\}$ (the set cl$_{X\times Y}$Gr$(T)$ is
called the adherence of the graph of T)$.$

It is easy to see that cl$T(x)\subset \overline{T}(x)$ for each $x\in
X.\medskip $

\begin{lemma}
(see [23]) \textit{Let }$X$\textit{\ be a topological space, }$Y$%
\textit{\ be a non-empty subset of a topological vector space E, \ss\ be a
base of the neighborhoods of 0 in E and }$A:X\rightarrow 2^{Y}.$\textit{\
For each }$V\in $\textit{\ss , let }$A_{V}:X\rightarrow 2^{Y}$\textit{\ be
defined by }$A_{V}(x)=(A(x)+V)\cap Y$\textit{\ for each }$x\in X.$\textit{\
If }$\widehat{x}\in X$\textit{\ and }$\widehat{y}\in Y$\textit{\ are such
that }$\widehat{y}\in \cap _{V\in \text{\ss }}\overline{A_{V}}(\widehat{x}),$%
\textit{\ then }$\widehat{y}\in \overline{A}(\widehat{x}).\medskip $
\end{lemma}

\begin{definition}
Let $X$, $Y$ be topological spaces and $T:X\rightarrow 2^{Y}$ be a
correspondence
\end{definition}

\QTP{Body Math}
1. $T$ is said to be \textit{upper semicontinuous} if for each $x\in X$ and
each open set $V$ in $Y$ with $T(x)\subset V$, there exists an open
neighborhood $U$ of $x$ in $X$ such that $T(y)\subset V$ for each $y\in U$.

2. $T$ is said to be \textit{lower semicontinuous} if for each x$\in X$ and
each open set $V$ in $Y$ with $T(x)\cap V\neq \emptyset $, there exists an
open neighborhood $U$ of $x$ in $X$ such that $T(y)\cap V\neq \emptyset $
for each $y\in U$.

3. $T$ is said to have \textit{open lower sections} if $T^{-1}(y):=\{x\in
X:y\in T(x)\}$ is open in $X$ for each $y\in Y.$\medskip

\begin{lemma}
(see [24]).\textit{Let }$X$\textit{\ be a topological space, }$Y$\textit{\
be a topological linear space, and let }$A:X\rightarrow 2^{Y}$\textit{\ be
an upper semicontinuous correspondence with compact values. Assume that the
sets }$C\subset Y$\textit{\ and }$K\subset Y$\textit{\ are closed and
respectively compact. Then }$T:X\rightarrow 2^{Y}$\textit{\ defined by }$%
T(x)=(A(x)+C)\cap K$\textit{\ for all }$x\in X$\textit{\ is upper
semicontinuous.}\medskip 
\end{lemma}

We present the following types of generalized convex functions and
correspondences.\medskip

\begin{definition}
(see [15]) Let $X$ be a convex set in a real vector space, and let $Z$ be
an ordered t.v.s, with a pointed convex cone $C.$ A vector-valued $%
f:X\rightarrow Z$ is said to be \emph{natural quasi }$C-$\emph{convex} on $X$
if $f(\lambda x_{1}+(1-\lambda )x_{2})\in $co$\{f(x_{1}),f(x_{2})\}-C$ for
every $x_{1},x_{2}\in X$ and $\lambda \in \lbrack 0,1].$ This condition is
equivalent with the following condition: there exists $\mu \in \lbrack 0,1]$
such that $f(\lambda x_{1}+(1-\lambda )x_{2})\leq _{C}\mu f(x_{1})+(1-\mu
)f(x_{2}),$ where $x\leq _{C}y$ $\Leftrightarrow $ $y-x\in C.$
\end{definition}

A vector-valued function f is said to be \emph{natural quasi }$C-$\emph{%
concave} on $X$ if $-f$ is natural quasi $C-$convex on $X$.$\medskip $

\begin{definition}
(see [26]) Let $E_{1}$, $E_{2}$ and $Z$ be real Hausdorff topological
vector spaces, $C\subset Z$ be a closed convex pointed cone with int$S\neq
\emptyset $; let $X$ be a nonempty convex subset of $E_{1}$, $T:X\rightarrow
2^{Z}$ be a correspondence. $T$ is said to be \emph{naturally
C-quasi-concave }on $X$, if for any $x_{1}$,$x_{2}\in X$ and $\lambda \in
\lbrack 0,1],$ co$(T(x_{1}),T(x_{2}))\subset $ $T(\lambda x_{1}+(1-\lambda
)x_{2})-C$.
\end{definition}

Let $\Delta _{n-1}=\left\{ (\lambda _{1},\lambda _{2},...,\lambda _{n})\in 
\mathbb{R}^{n}:\overset{n}{\underset{i=1}{\tsum }}\lambda _{i}=1\text{ and }%
\lambda _{i}\geqslant 0,i=1,2,...,n\right\} $ be the standard
(n-1)-dimensional simplex in $\mathbb{R}^{n}.$

\begin{definition}
(see [5])\textit{\ }Let $X$ be a non-empty convex subset of a topological
vector space $E$ \ and $Y$ be a non-empty subset of $E$.\textit{\ }The
correspondence\textit{\ }$T:X\longrightarrow 2^{Y}$ is said to have \emph{%
weakly convex graph} (in short, it is a WCG correspondence) if for each
finite set $\{x_{1},x_{2},...,x_{n}\}\subset X$, there exists $y_{i}\in
T(x_{i})$, $(i=1,2,...,n)$ such that
\end{definition}

\begin{center}
(1.1) $\ $co$(\{(x_{1},y_{1}),(x_{2},y_{2}),...,(x_{n},y_{n})\})\subset $Gr$%
(T)$
\end{center}

\QTP{Body Math}
The relation (1.1) is equivalent to

\begin{center}
(1.2) $\ \ \ \ \overset{n}{\underset{i=1}{\tsum }}\lambda _{i}y_{i}\in T(%
\overset{n}{\underset{i=1}{\tsum }}\lambda _{i}x_{i})$ \ \ \ \ \ \ \ $%
(\forall (\lambda _{1},\lambda _{2},...,\lambda _{n})\in \Delta _{n-1}).$%
\medskip
\end{center}

We introduce the concept of weakly naturally quasi-concave correspondence.

\begin{definition}
Let $X$ be a nonempty convex subset of a topological vector space $E$ \ and $%
Y$ be a nonempty subset of a topological vector space $F$.\textit{\ }The
correspondence\textit{\ }$T:X\longrightarrow 2^{Y}$ is said to be \emph{%
weakly naturally quasi-concave (WNQ) }if for each $n$ and for each finite
set $\{x_{1},x_{2},...,x_{n}\}\subset X$, there exists $y_{i}\in T(x_{i})$, $%
(i=1,2,...,n)$ and $g=(g_{1},g_{2},...,g_{n}):\Delta _{n-1}\rightarrow
\Delta _{n-1}$ a function with $g_{i}$ continuous, $g_{i}(1)=1$ and $%
g_{i}(0)=0$ for each\textit{\ }$i=1,2,...,n$, such that for every $(\lambda
_{1},\lambda _{2},...,\lambda _{n})\in \Delta _{n-1}$, there exists $y=%
\overset{n}{\underset{i=1}{\tsum }}g_{i}(\lambda _{i})y_{i}\in T(\overset{n}{%
\underset{i=1}{\tsum }}\lambda _{i}x_{i}).$
\end{definition}

\begin{remark}
If $g_{i}(\lambda _{i})=\lambda _{i}$ for each $i\in (1,2,...,n)$ and $%
(\lambda _{1},\lambda _{2},...,\lambda _{n})\in \Delta _{n-1},$ we get a
correspondence with weakly convex graph, as it is defined by Ding and He
Yiran in [5]. In the same time, the weakly naturally quasi-concavity is a
weakening of the notion of naturally C-quasi-concavity with $C=\{0\}.$
\end{remark}

\begin{remark}
If $T$ is a single valued mapping, then it must be natural quasi $C$-concave
for $C=\{0\}.$\medskip 
\end{remark}

\begin{example}
Let $T:[0,4]\rightarrow 2^{[-2,2]}$ be defined by $T(x)=\left\{ 
\begin{array}{c}
\lbrack 0,2]\text{ if }x\in \lbrack 0,2); \\ 
\lbrack -2,0]\text{ \ \ if \ }x=2; \\ 
(0,2]\text{ if }x\in (2,4].%
\end{array}%
\right. $
\end{example}

$T$ is neither upper semicontinuous, nor lower semicontinuous in $2.$ $T$
also has not weakly convex graph, since if we consider $n=2,$ $x_{1}=1$ and $%
x_{2}=3,$ we have that co$\{(1,y_{1}),(3,y_{2})\}\nsubseteq $Gr$T$ for every 
$y_{1}\in T(x_{1}),y_{2}\in T(x_{2}).$

We shall prove that $T$ is a weakly naturally quasi-concave correspondence.

1) Let's consider first $n=2.$

\ \ \ a) If $x_{1},x_{2}\in \lbrack 0,2)$ and $x_{1},x_{2}\in (2,4],$ there
exists $y_{1}=2\in T(x_{1})$, $y_{2}=2\in $

$\ \ \ T(x_{2})$ and $g_{i}(\lambda _{i})=\lambda _{i},$ $i=1,2$ such that
for each $(\lambda _{1},\lambda _{2})$ with the property

\ \ \ \ that $\lambda _{1}\geq 0,$ $\lambda _{2}\geq 0$, $\lambda
_{1}+\lambda _{2}=1,$ there exists $y=\overset{2}{\underset{i=1}{\tsum }}%
g_{i}(\lambda _{i})y_{i}\in T(\overset{2}{\underset{i=1}{\tsum }}\lambda
_{i}x_{i}).$

\ \ \ b) If $x_{1}\in \lbrack 0,2)$ and $x_{2}\in (2,4],$ there exists $%
\lambda _{1}^{\ast }\neq 0$ such that $\lambda _{1}^{\ast }x_{1}+(1-\lambda
_{2}^{\ast })x_{2}=$

$\ \ \ \ \ =2.$

\ \ Let's consider $g_{i}:[0,1]\rightarrow \lbrack 0,1]$ continuous
functions such that $g_{i}(1)=1,$

$\ \ g_{i}(0)=0$ for each\textit{\ }$i=1,2$ and $g_{1}(\lambda
_{1})+g_{2}(\lambda _{2})=1$ if $\lambda _{1}+\lambda _{2}=1,$

\ \ defined by

$\ \ g_{1}(\lambda _{1})=\left\{ 
\begin{array}{c}
\frac{1}{\lambda _{1}^{\ast }}\lambda _{1}\text{ if }\lambda _{1}\in \lbrack
0,\lambda _{1}^{\ast }); \\ 
1\text{ \ \ \ if \ \ \ }\lambda _{1}\in \lbrack \lambda _{1}^{\ast },1]%
\end{array}%
\right. $ and

$\ \ g_{2}(\lambda _{2})=\left\{ 
\begin{array}{c}
0\text{ \ \ \ \ \ \ \ \ \ if \ \ \ \ \ \ \ \ \ }\lambda _{2}\in \lbrack
0,1-\lambda _{1}^{\ast }]; \\ 
1-\frac{1}{\lambda _{1}^{\ast }}+\frac{1}{\lambda _{1}^{\ast }}\lambda _{2}%
\text{ if }\lambda _{2}\in (1-\lambda _{1}^{\ast },1].%
\end{array}%
\right. $

\ \ There exists $y_{1}=0$ and $y_{2}=2$ such that

\ \ \ \ \ b1) for $\lambda _{1}\in \lbrack 0,\lambda _{1}^{\ast })$ and $%
\lambda _{2}=1-\lambda _{1},$ $x=\lambda _{1}x_{1}+\lambda _{2}x_{2}\in
(2,x_{2}],$ then

$\ \ \ \ \ T(x)=(0,2]$ and $y=g_{1}(\lambda _{1})y_{1}+g_{2}(\lambda
_{2})y_{2}=\frac{1}{\lambda _{1}^{\ast }}\lambda _{1}y_{1}+(1-\frac{1}{%
\lambda _{1}^{\ast }}\lambda _{1})y_{2}=$

$\ \ \ \ \ (1-\frac{1}{\lambda _{1}^{\ast }}\lambda _{1})2$ $\in
(0,2]=T(\lambda _{1}x_{1}+\lambda _{2}x_{2});$

\ \ \ \ \ b2) for $\lambda _{1}\in (\lambda _{1}^{\ast },1],$ and $\lambda
_{2}=1-\lambda _{1},$ $x=\lambda _{1}x_{1}+\lambda _{2}x_{2}\in \lbrack
x_{1},2),$ then

$\ \ \ \ \ T(x)=[0,2]$ and $y=g_{1}(\lambda _{1})y_{1}+g_{2}(\lambda
_{2})y_{2}=1\times 0+0\times 2=0\in $

$\ \ \ \ \ \in T(\lambda _{1}x_{1}+\lambda _{2}x_{2});$

\ \ \ \ \ b3) If $\lambda _{1}=\lambda _{1}^{\ast },$ $\lambda
_{2}=1-\lambda _{1}^{\ast },$ $x=\lambda _{1}x_{1}+\lambda _{2}x_{2}=2,$
then $T(x)=[-2,0]$

\ \ \ \ \ and $y=g_{1}(\lambda _{1})y_{1}+g_{2}(\lambda _{2})y_{2}=1\times
0+0\times 2=0\in T(2);$

\ \ c) If $x_{1}\in \lbrack 0,2)$ and $x_{2}=2$, there exists $y_{1}=2$, $%
y_{2}=0$ and the continuous

\ \ functions $g_{i}:[0,1]\rightarrow \lbrack 0,1]$ with $g_{i}(1)=1$, $%
g_{i}(0)=0$ for each\textit{\ }$i=1,2$ and

$\ \ g_{1}(\lambda _{1})+g_{2}(\lambda _{2})=1$ if $\lambda _{1}+\lambda
_{2}=1$ such that

\ \ \ \ \ c1) for $\lambda _{1}\in (0,1]$ and $\lambda _{2}=1-\lambda _{1},$ 
$x=\lambda _{1}x_{1}+\lambda _{2}x_{2}\in \lbrack x_{1},x_{2})$, then

$\ \ \ \ \ T(x)=[-2,0]$ and $y=g_{1}(\lambda _{1})y_{1}+g_{2}(\lambda
_{2})y_{2}=g_{1}(\lambda _{1})\times 2+g_{2}(\lambda _{2})\times 0$

$\ \ \ \ \ =g_{1}(\lambda _{1})\times 2\in T(x);$

\ \ \ \ \ c2) for $\lambda _{1}=0$ and $\lambda _{2}=1,$ $x=2,$ then $%
T(2)=[-2.0]$ and

$\ \ \ \ \ y=g_{1}(0)y_{1}+g_{2}(1)y_{2}=0\times 2+1\times 0=0\in T(2);$

\ \ d) If $x_{1}=2$ and $x_{2}\in (2,4],$ there exists $y_{1}=0$, $y_{2}=2$
and the continuous

\ \ functions $g_{i}:[0,1]\rightarrow \lbrack 0,1]$ with $g_{i}(1)=1$, $%
g_{i}(0)=0$ for each\textit{\ }$i=1,2$ and

$\ \ g_{1}(\lambda _{1})+g_{2}(\lambda _{2})=1$ if $\lambda _{1}+\lambda
_{2}=1$ such that

\ \ \ \ \ d1) for $\lambda _{1}=1$ and $\lambda _{2}=0,$ $x=2,$ then $%
T(2)=[-2.0]$ and

$\ \ \ \ \ \ y=g_{1}(1)y_{1}+g_{2}(0)y_{2}=1\times 0+0\times 2=0\in T(2);$

\ \ \ \ \ \ d2) for $\lambda _{1}\in \lbrack 0,1)$ and $\lambda
_{2}=1-\lambda _{1},$ $x=\lambda _{1}x_{1}+\lambda _{2}x_{2}\in
(x_{1},x_{2}] $, then

$\ \ \ \ \ \ T(x)=(0,2]$ and $y=g_{1}(\lambda _{1})y_{1}+g_{2}(\lambda
_{2})y_{2}=g_{1}(\lambda _{1})\times 0+g_{2}(\lambda _{2})\times 2$

$\ \ \ \ \ \ =g_{2}(\lambda _{2})\times 2$ $\in (0,2]=T(x).$

2) The case $n>2$ can be reduced to the case 1).$\medskip $

Now, we introduce the following definitions.$\medskip $

Let $I$ be an index set. For each $i\in I$, let $X_{i}$ be a non-empty
convex subset of a topological linear space $E_{i}$ and denote $X=\underset{%
i\in I}{\tprod }X_{i}$.

\begin{definition}
\textit{Let }$K_{i}$ be a subset of $X$. The correspondence $%
A_{i}:X\rightarrow 2^{X_{i}}$ is said to have the WNQS\textit{-property} on $%
K_{i}$, if there is a weakly naturally quasi-concave correspondence $%
T_{i}:K_{i}\rightarrow 2^{X_{i}}$ such that $x_{i}\notin T_{i}(x)$ and $%
T_{i}(x)\subset A_{i}(x)$ for all $x\in K_{i}.\medskip $
\end{definition}

\begin{definition}
\textit{Let }$K_{i}$ be a subset of $X$. The correspondence $%
A_{i}:X\rightarrow 2^{X_{i}}$ is said to have the e-WNQS\textit{-property}
on $K_{i}$ if for each convex neighborhood $V$ of $0$ in $X_{i},$ there is a
weakly naturally quasi-concave correspondence $T_{i}^{V}:K_{i}\rightarrow
2^{X_{i}}$ such that $x_{i}\notin T_{i}^{V}(x)$ and $T_{i}^{V}(x)\subset
A_{i}(x)+V$ for all $x\in K_{i}.\medskip $
\end{definition}

\begin{definition}
Let $X$ be a nonempty convex subset of a topological vector space $E$ \ and $%
Y$ be a nonempty subset of a topological vector space $F$.\textit{\ }The
correspondence\textit{\ }$T:X\longrightarrow 2^{Y}$ is said to be \emph{%
weakly *-concave } if for each $n$ and for each finite set $%
\{x_{1},x_{2},...,x_{n}\}\subset X$, there exists $y_{i}\in T(x_{i})$, $%
(i=1,2,...,n)$, such that for every $(\lambda _{1},\lambda _{2},...,\lambda
_{n})\in \Delta _{n-1}$, $\overset{n}{\underset{i=1}{\tsum }}\lambda
_{i}y_{i}\subset T(x)$, for each $x\in X.\medskip $
\end{definition}

\QTP{Body Math}
To prove our theorems of equilibrium existence, we need the
following:\medskip

\begin{theorem}
(Wu's fixed point theorem [19]) L\textit{et }$\mathit{I}$\textit{\ be an
index set. For each }$i\in I,$\textit{\ let }$X_{i}$\textit{\ be a nonempty
convex subset of a Hausdorff locally convex topological vector space }$%
E_{i}, $\textit{\ }$D_{i\text{ }}$\textit{a non-empty compact metrizable
subset of }$X_{i}$\textit{\ and }$S_{i},T_{i}:X:=\underset{i\in I}{\tprod }%
X_{i}\rightarrow 2^{D_{i}}$\textit{\ two correspondences with the following
conditions:}
\end{theorem}

(i) \textit{for each }$x\in X,$ clco$S_{i}(x)\subset T_{i}(x)$\textit{\ and }%
$S_{i}(x)\neq \emptyset ,$

(ii)\textit{\ }$S_{i}$\textit{\ is lower semicontinuous.}

\textit{Then, there exists a point }$\overline{x}=\underset{i\in I}{\tprod }%
x_{i}\in D=\underset{i\in I}{\tprod }D_{i}$\textit{\ such that }$\overline{x}%
_{i}\in T_{i}(\overline{x})$\textit{\ for each }$i\in I.\medskip $

The extension of Kakutani's theorem on locally convex spaces is due to Ky
Fan.

\begin{theorem}
(Ky-Fan, [7]) \textit{Let }$Y$\textit{\ be a locally convex space, }$%
X\subset Y$\textit{\ be a compact and convex subset and }$T:X\rightarrow
2^{X}$\textit{\ be an upper semicontinuous correspondence with non-empty
compact convex values. Then, }$T$\textit{\ has a fixed point.\medskip }
\end{theorem}

For the case when $X$ is not compact, Himmelberg got the following result.

\begin{theorem}
(Himmelberg, [8]) \textit{Let }$X$\textit{\ be a non-empty convex subset of
a separated locally convex space }$Y$\textit{. Let }$T:X\rightarrow 2^{X}$ 
\textit{be an upper semicontinuous correspondence such that }$T(x)$\textit{\
is closed and convex for each }$x\in X$\textit{, and }$T(X)$\textit{\ is
contained in a compact subset }$C$\textit{\ of }$X$\textit{. Then, }$T$%
\textit{\ has a fixed point.}
\end{theorem}

\section{\textbf{FIXED POINT THEOREMS}}

We formulate the following fixed point theorem for weakly naturally
quasi-concave correspondences.\medskip

\begin{theorem}
(selection theorem) \textit{Let }$Y$\textit{\ be a non-empty subset of a
topological vector space }$E$\textit{\ and }$K$\textit{\ be a }$(n-1)$%
\textit{- dimensional simplex in a topological vector space }$F.$ \textit{\
Let }$T:K\rightarrow 2^{Y}$\textit{\ be a weakly naturally quasi-concave
correspondence. Then, }$T$\textit{\ has a continuous selection on }$K$%
\textit{.\medskip }
\end{theorem}

\begin{proof}
Let $a_{1},a_{2},...,a_{n}$ be the vertices of $K.$ Since $T$ is weakly
naturally quasi-concave, there exist $b_{i}\in T(a_{i})$, $(i=1,2,...,n)$
and $g=(g_{1},g_{2},...,g_{n}):\Delta _{n-1}\rightarrow \Delta _{n-1}$ a
function with $g_{i}$ continuous, $g_{i}(1)=1$ and $g_{i}(0)=0$ for each%
\textit{\ }$i=1,2,...n$, such that for every $(\lambda _{1},\lambda
_{2},...,\lambda _{n})\in \Delta _{n-1}$, there exists $y\in T(\overset{n}{%
\underset{i=1}{\tsum }}\lambda _{i}a_{i})$ with $y=\overset{n}{\underset{i=1}%
{\tsum }}g_{i}(\lambda _{i})y_{i}.$

Since $K$ is a $(n-1)$-dimensional simplex with the vertices $%
a_{1},...,a_{n},$ there exists unique continuous functions $\lambda
_{i}:K\rightarrow \mathbb{R},$ $i=1,2,...,n$ such that for each $x\in K,$ we
have $(\lambda _{1}(x),\lambda _{2}(x),...,\lambda _{n}(x))\in \Delta _{n-1}$
and $x=\overset{n}{\underset{i=1}{\tsum }}\lambda _{i}(x)a_{i}.$

Let's define $f:K\rightarrow 2^{Y}$ by

$f(a_{i})=b_{i}$ $(i=1,...,n)$ and

$f(\overset{n}{\underset{i=1}{\tsum }}\lambda _{i}a_{i})=\overset{n}{%
\underset{i=1}{\tsum }}g_{i}(\lambda _{i})b_{i}\in T(x).$

We show that $f$ is continuous.

Let $(x_{m})_{m\in N}$ be a sequence which converges to $x_{0}\in K,$ where $%
x_{m}=\overset{n}{\underset{i=1}{\tsum }}\lambda _{i}(x_{m})a_{i}$ and $%
x_{0}=$ $\overset{n}{\underset{i=1}{\tsum }}\lambda _{i}(x_{0})a_{i}.$ By
the continuity of $\lambda _{i},$ it follows that for each $i=1,2,...,n$, $%
\lambda _{i}(x_{m})\rightarrow \lambda _{i}(x_{0})$ as $m\rightarrow \infty
. $ Since $\ g_{1},...,g_{n}$ are continuous, we have $g_{i}(\lambda
_{i}(x_{m}))\rightarrow g_{i}(\lambda _{i}(x_{0}))$ as $m\rightarrow \infty
. $ Hence $f(x_{m})\rightarrow f(x_{0})$ as $m\rightarrow \infty ,$ i.e. $f$
is continuous.\medskip
\end{proof}

\begin{theorem}
\textit{Let }$Y$\textit{\ be a non-empty subset of a topological vector
space }$E$\textit{\ and }$K$\textit{\ be a }$(n-1)$\textit{- dimensional
simplex in }$E$\textit{. Let }$T:K\rightarrow 2^{Y}$\textit{\ be an weakly
naturally quasi-concave correspondence and }$s:Y\rightarrow K$\textit{\ be a
continuous function. Then, there exists }$x^{\ast }\in K$\textit{\ such that
\ }$x^{\ast }\in s\circ T(x^{\ast })$\textit{.\medskip }
\end{theorem}

\begin{proof}
By Theorem 6, $T$ has a continuous selection theorem on $K.$ Since $%
s:Y\rightarrow K$ is continuous, we obtain that $s\circ f:K\rightarrow K$ is
continuous. By Brouwer's fixed point theorem, there exists a point $x^{\ast
}\in K$ such that $x^{\ast }=s\circ f(x^{\ast })$ and then, $x^{\ast }\in
s\circ T(x^{\ast }).$\medskip
\end{proof}

\begin{theorem}
(selection theorem). \textit{Let }$Y$\textit{\ be a non-empty subset of a
topological vector space }$E$\textit{\ and }$K$\textit{\ be a }$(n-1)$%
\textit{- dimensional simplex in a topological vector space }$F.$ \textit{\
Let }$T:K\rightarrow 2^{Y}$\textit{\ be a weakly *-concave correspondence.
Then, }$T$\textit{\ has a continuous selection on }$K$\textit{.\medskip }
\end{theorem}

\begin{proof}
Let $a_{1},a_{2},...,a_{n}$ be the vertices of $K.$ Since $T$ is weakly
*-concave, there exist $b_{i}\in T(a_{i})$, $(i=1,2,...,n)$ such that for
every $(\lambda _{_{1}},\lambda _{_{2}},...,\lambda _{n})\in \Delta _{n-1}$, 
$\overset{n}{\underset{i=1}{\tsum }}\lambda _{i}b_{i}\subset T(x)$, for each 
$x\in X.$

Since $K$ is a $(n-1)$-dimensional simplex with the vertices $%
a_{1},...,a_{n},$ there exists unique continuous functions $\lambda
_{i}:K\rightarrow \mathbb{R},$ $i=1,2,...,n$ such that for each $x\in K,$ we
have $(\lambda _{1}(x),\lambda _{2}(x),...,\lambda _{n}(x))\in \Delta _{n-1}$
and $x=\overset{n}{\underset{i=1}{\tsum }}\lambda _{i}(x)a_{i}.$

Let's define $f:K\rightarrow 2^{Y}$ by

$f(a_{i})=b_{i}$ $(i=1,...,n)$ and

$f(\overset{n}{\underset{i=1}{\tsum }}\lambda _{i}a_{i})=\overset{n}{%
\underset{i=1}{\tsum }}\lambda _{i}b_{i}\in T(x).$

We show that $f$ is continuous.

Let $(x_{m})_{m\in N}$ be a sequence which converges to $x_{0}\in K$ where $%
x_{m}=\overset{n}{\underset{i=1}{\tsum }}\lambda _{i}(x_{m})a_{i}$ and $%
x_{0}=$ $\overset{n}{\underset{i=1}{\tsum }}\lambda _{i}(x_{0})a_{i}.$ By
the continuity of $\lambda _{i},$ it follows that for each $i=1,2,...,n$, $%
\lambda _{i}(x_{m})\rightarrow \lambda _{i}(x_{0})$ as $m\rightarrow \infty
. $ Hence we must have $f(x_{m})\rightarrow f(x_{0})$ as $m\rightarrow
\infty , $ i.e. $f$ is continuous.\medskip
\end{proof}

\begin{theorem}
\textit{Let }$Y$\textit{\ be a non-empty subset of a topological vector
space }$E$\textit{\ and }$K$\textit{\ be a }$(n-1)$\textit{- dimensional
simplex in }$E.$\textit{\ Let }$T:K\rightarrow 2^{Y}$\textit{\ be a weakly
*-concave correspondence\ and }$s:Y\rightarrow K$\textit{\ be a continuous
function. Then, there exists }$x^{\ast }\in K$\textit{\ such that \ }$%
x^{\ast }\in s\circ T(x^{\ast })$\textit{.\medskip }
\end{theorem}

\begin{proof}
By Theorem 8, $T$ has a continuous selection theorem on $K.$ Since $%
s:Y\rightarrow K$ is continuous, we obtain that $s\circ f:K\rightarrow K$ is
continuous. By Brouwer's fixed point theorem, there exists a point $x^{\ast
}\in K$ such that $x^{\ast }=s\circ f(x^{\ast })$ and then, $x^{\ast }\in
s\circ T(x^{\ast }).$\medskip
\end{proof}

\section{\textbf{EQUILIBRIUM THEOREMS}}

First, we present the model of an abstract economy and the definition of an
equilibrium.

Let $I$ be a non-empty set (the set of agents). For each $i\in I$, let $%
X_{i} $ be a non-empty topological vector space representing the set of
actions and define $X:=\underset{i\in I}{\prod }X_{i}$; let $A_{i}$, $%
B_{i}:X\rightarrow 2^{X_{i}}$ be the constraint correspondences and $P_{i}$
the preference correspondence.

\begin{definition}
The family $\Gamma =(X_{i},A_{i},P_{i},B_{i})_{i\in I}$ is said to be an 
\textit{abstract economy.}
\end{definition}

\begin{definition}
An \textit{equilibrium} for $\Gamma $ is defined as a point $\overline{x}\in
X$ such that for each $i\in I$, $\overline{x}_{i}\in \overline{B}_{i}(%
\overline{x})$ and $A_{i}(\overline{x},)\cap P_{i}(\overline{x})=\emptyset $%
.\medskip 
\end{definition}

\begin{remark}
When for each $i\in I$, $A_{i}(x)=B_{i}(x)$ for all $x\in X,$ this abstract
economy model coincides with the classical one introduced by Borglin and
Keiding in [2]. If in addition, $\overline{B}_{i}(\overline{x})=$cl$%
_{X_{i}}B_{i}(\overline{x})$ for each $x\in X,$ which is the case if $B_{i}$
has a closed graph in $X\times X_{i}$, the definition of an equilibrium
coincides with the one used by Yannelis and Prabhakar [21].\medskip 
\end{remark}

To prove the following theorems we use the selection theorem mentioned in
Section 3. We show the existence of equilibrium for an abstract economy
without assuming the continuity of the constraint and the preference
correspondences \textit{\ }$A_{i}$ and $P_{i}$.

First, we prove a new equilibrium existence theorem for a noncompact
abstract economy with constraint and preference correspondences $A_{i}$ and $%
P_{i},$ which have the property that their intersection $A_{i}\cap P_{i}$
contains a WNQ selector on the domain $W_{i}$ of $A_{i}\cap P_{i}$ and $%
W_{i} $ must be a simplex$.$ To find the equilibrium point, we use Wu's
fixed point theorem [19].

Since the constraint correspondence $B_{i}$ is lower semicontinuous for each 
$i\in I,$ the next theorem can be compared with Theorem 5 of Wu [19]. The
proofs of these results are based on similar methods.\medskip

\begin{theorem}
\textit{Let }$\Gamma =(X_{i},A_{i},P_{i},B_{i})_{i\in I}$\textit{\ \ be an
abstract economy, where }$I$\textit{\ is a (possibly uncountable) set of
agents such that for each }$i\in I:$
\end{theorem}

(1)\textit{\ }$X_{i}$\textit{\ is a non-empty convex set in a locally convex
space }$E_{i}$ \textit{and there exists a compact subset }$D_{i}$\textit{\
of }$X_{i}$\textit{\ containing all the values of the correspondences }$%
A_{i},P_{i}$\textit{\ and }$B_{i}$\textit{\ such that }$D=\tprod\limits_{i%
\in I}D_{i}$\textit{\ is metrizable;}

(2)\textit{\ }cl$B_{i}$\textit{\ is lower semicontinuous, has non-empty
convex values and for each }$x\in X$,\textit{\ }$A_{i}(x)\subset B_{i}(x)$%
\textit{;}

(3)\textit{\ } $W_{i}\mathit{\ }=\left\{ x\in X\text{ / }\left( A_{i}\cap
P_{i}\right) (x)\neq \emptyset \right\} $\textit{\ is} \textit{a }$(n_{i}-1)$%
\textit{-dimensional simplex in }$X$\textit{\ such that }$W_{i}\subset $co$D$%
\textit{;}

(4)\textit{\ there exists a weakly naturally quasi-concave correspondence }$%
S_{i}:W_{i}\rightarrow 2^{D_{i}}$\textit{\ such that}$\mathit{\ }$ $%
S_{i}(x)\subset \left( A_{i}\cap P_{i}\right) (x)$ for each $x\in W_{i}$%
\textit{;}

(5)\textit{\ for each }$x\in W_{i},$\textit{\ }$x_{i}\notin (A_{i}\cap
P_{i})(x)$\textit{.}

\textit{Then, there exists an equilibrium point }$\overline{x}\in D$ \textit{%
\ for }$\Gamma $\textit{,}$\ i.e.$\textit{, for each }$i\in I$\textit{, }$%
\overline{x}_{i}\in $cl$B_{i}(\overline{x})$\textit{\ and }$A_{i}(\overline{x%
})\cap P_{i}(\overline{x})=\emptyset $\textit{.\medskip }

\begin{proof}
Let be $i\in I.$ From the assumption (4) and the selection theorem (Theorem
6), it follows that there exists a continuous function $f_{i}:W_{i}%
\rightarrow D_{i}$ such that for each $x\in W_{i}$, $f_{i}(x)\in
S_{i}(x)\subset A_{i}(x)\cap P_{i}(x)\subset B_{i}(x).$

Define the correspondence $T_{i}:X\rightarrow 2^{D_{i}}$, by $%
T_{i}(x):=\left\{ 
\begin{array}{c}
\{f_{i}(x)\}\text{, if }x\in W_{i}\text{, } \\ 
\text{cl}B_{i}(x)\text{, if }x\notin W_{i}\text{.}%
\end{array}%
\right. $

$T_{i}$ is lower semicontinuous on $X$.

Let $V$ be a closed subset of $\ X_{i}$, then

$U:=\{x\in X$ $\mid $ $T_{i}(x)\subset V\}$\ \ \ =$\{x\in W_{i}$ $\mid $ $%
T_{i}(x)\subset V\}\cup \{x\in X\setminus W_{i}$ $\mid $ $T_{i}(x)\subset
V\} $

\ \ \ \ \ =$\left\{ x\in W_{i}\text{ }\mid \text{ }f_{i}(x)\in V\right\}
\cup \left\{ x\in X\mid \text{ cl}B_{i}(x)\subset V\right\} $

\ \ \ \ \ =$(f_{i}^{-1}(V)\cap W_{i})\cup \left\{ x\in X\mid \text{ cl}%
B_{i}(x)\subset V\right\} .$

$U$ is a closed set, because $W_{i}$ is closed, $f_{i}$ is a continuous
function on int$_{X}K_{i}$ and the set $\left\{ x\in X\mid \text{ cl}%
B_{i}(x)\subset V\right\} $ is closed since cl$B_{i}$ is l.s.c. Let $%
D=\tprod\limits_{i\in I}D_{i}.$ Then by Tychonoff's Theorem, $D$ is compact
in the convex set $X$.

By Theorem 3 (Wu's fixed-point theorem), applied for the correspondences $%
S_{i}=T_{i}$ and $T_{i}:X\rightarrow 2^{D_{i}},$ there exists $\overline{x}%
\in D$ such that for each $i\in I$, $\overline{x}_{i}\in T_{i}(\overline{x})$%
. If $\overline{x}\in W_{i}$ for some $i\in I$, then $\overline{x}_{i}=f_{i}(%
\overline{x})$, which is a contradiction.

Therefore, $\overline{x}\notin W_{i}$, and hence, $(A_{i}\cap P_{i})(%
\overline{x})=\emptyset $. Also, for each $i\in I$, we have $\overline{x}%
_{i}\in T_{i}(\overline{x})$, and then, $\overline{x}_{i}\in $cl$B_{i}(%
\overline{x}).\medskip $
\end{proof}

\begin{remark}
In this theorem, the correspondences $A_{i}\cap P_{i},$ $i\in I,$ may not
verify continuity assumptions and may not have convex or compact
values.\medskip
\end{remark}

\begin{remark}
In assumption (3), $W_{i}$ must be a proper subset of $X$. In fact, if $%
W_{i}=X_{i}$, then, by applying Himmelberg's fixed point theorem ([8]) to $%
\underset{i\in I}{\prod }f_{i}(x),$ where $f_{i}$ is a continuous selection
of $S_{i}\subset A_{i}\cap P_{i}$, we can get a fixed point $\overline{x}\in 
\underset{i\in I}{\prod }(A_{i}\cap P_{i})(\overline{x})$, which contradicts
assumption (5).\medskip
\end{remark}

Since a correspondence $T:X\rightarrow 2^{Y}$ having the property that $\cap
\{T(x):x\in X\}$ is nonempty and convex, is a WNQ correspondence, we obtain
the following corollary.\medskip

\begin{corollary}
\textit{Let }$\Gamma =(X_{i},A_{i},P_{i},B_{i})_{i\in I}$\textit{\ \ be an
abstract economy, where }$I$\textit{\ is a (possibly uncountable) set of
agents such that for each }$i\in I:$
\end{corollary}

(1)\textit{\ }$X_{i}$\textit{\ is a non-empty convex set in a locally convex
space }$E_{i}$ \textit{and there exists a compact subset }$D_{i}$\textit{\
of }$X_{i}$\textit{\ containing all the values of the correspondences }$%
A_{i},P_{i}$\textit{\ and }$B_{i}$\textit{\ such that }$D=\tprod\limits_{i%
\in I}D_{i}$\textit{\ is metrizable;}

(2)\textit{\ }cl$B_{i}$\textit{\ is lower semicontinuous, has non-empty
convex values and for each }$x\in X$,\textit{\ }$A_{i}(x)\subset B_{i}(x)$%
\textit{;}

(3)\textit{\ } $W_{i}\mathit{\ }=\left\{ x\in X\text{ / }\left( A_{i}\cap
P_{i}\right) (x)\neq \emptyset \right\} $\textit{\ is a }$(n_{i}-1)$\textit{%
-dimensional simplex in }$X$\textit{\ such that }$W_{i}\subset $co$D$\textit{%
;}

(4)\textit{\ there exists a correspondence }$S_{i}:W_{i}\rightarrow
2^{D_{i}} $\textit{\ such that}$\mathit{\ }S_{i}$ \textit{has the property
that }$\cap \{T(x):x\in X\}$ \textit{is nonempty and convex, and} $%
S_{i}(x)\subset \left( A_{i}\cap P_{i}\right) (x)$ \textit{for each }$x\in
W_{i}$\textit{;}

(5)\textit{\ for each }$x\in W_{i},$\textit{\ }$x_{i}\notin (A_{i}\cap
P_{i})(x)$\textit{.}

\textit{Then there exists an equilibrium point }$\overline{x}\in D$ \textit{%
\ for }$\Gamma $\textit{,}$\ i.e.$\textit{, for each }$i\in I$\textit{, }$%
\overline{x}_{i}\in $cl$B_{i}(\overline{x})$\textit{\ and }$A_{i}(\overline{x%
})\cap P_{i}(\overline{x})=\emptyset $\textit{.\medskip }

A correspondence $T:X\rightarrow 2^{Y}$ with convex graph is a WNQ
correspondence, and then we have:

\begin{corollary}
\textit{Let }$\Gamma =(X_{i},A_{i},P_{i},B_{i})_{i\in I}$\textit{\ \ be an
abstract economy, where }$I$\textit{\ is a (possibly uncountable) set of
agents such that for each }$i\in I:$
\end{corollary}

(1)\textit{\ }$X_{i}$\textit{\ is a non-empty compact convex set in a
locally convex space }$E_{i}$;

(2)\textit{\ }cl$B_{i}$\textit{\ is lower semicontinuous, has non-empty
convex values and for each }$x\in X$,\textit{\ }$A_{i}(x)\subset B_{i}(x)$%
\textit{;}

(3)\textit{\ } $W_{i}\mathit{\ }=\left\{ x\in X\text{ / }\left( A_{i}\cap
P_{i}\right) (x)\neq \emptyset \right\} $\textit{\ is} a $(n_{i}-1)$-\textit{%
dimensional simplex in }$X$;

(4)\textit{there exists a correspondence }$S_{i}:W_{i}\rightarrow 2^{X_{i}}$%
\textit{\ with convex graph such that}$\mathit{\ }S_{i}(x)$

\noindent $\subset \left( A_{i}\cap P_{i}\right) (x)$ \textit{for each }$%
x\in W_{i}$\textit{;}

(5)\textit{\ for each }$x\in W_{i},$\textit{\ }$x_{i}\notin (A_{i}\cap
P_{i})(x)$\textit{.}

\textit{Then there exists an equilibrium point }$\overline{x}\in X$ \textit{%
\ for }$\Gamma $\textit{,}$\ i.e.$\textit{, for each }$i\in I$\textit{, }$%
\overline{x}_{i}\in $cl$B_{i}(\overline{x})$\textit{\ and }$A_{i}(\overline{x%
})\cap P_{i}(\overline{x})=\emptyset $\textit{.\medskip }

For Theorem 13, we use an approximation method, in the meaning that we
obtain, for each $i\in I,$ a continuous selection $f_{i}^{V_{i}}$ of $%
(A_{i}+V_{i})\cap P_{i},$ where $V_{i}$ is a convex neighborhood of $0$ in $%
X_{i}$. For every $V=\tprod\limits_{i\in I}V_{i}$, we obtain an equilibrium
point for the associated approximate abstract economy $\Gamma
_{V}=(X_{i},A_{i},P_{i},B_{V_{i}})_{i\in I}$\textit{, }i.e.\textit{,} a point%
\textit{\ }$\overline{x}\in X$ such that $A_{i}(\overline{x})\cap P_{i}(%
\overline{x})=\emptyset $ and $\overline{x}_{i}\in B_{V_{i}}(\overline{x}),$
where the correspondence $B_{V_{i}}:X\rightarrow 2^{X_{i}}$ is defined by $%
B_{V_{i}}(x)=$cl$(B_{i}(x)+V_{i})\cap X_{i}$ for each $x\in X$ and for each $%
i\in I.$ Finally, we use Lemma 1 to get an equilibrium point for $\Gamma $
in $X$. The compactness assumption for $X_{i}$ is essential in the proof.

Examples of results which use an approximation method are Theorem 3.1 pag 37
or Theorem 1.2, pag. 41 in [23]. This method is usually used in relation
with abstract economies which have lower semicontinuous constraint
correspondences.\medskip

\begin{theorem}
\textit{Let }$\Gamma =(X_{i},A_{i},P_{i},B_{i})_{i\in I}$\textit{\ \ be an
abstract economy, where }$I$\textit{\ is a (possibly uncountable) set of
agents such that for each }$i\in I:$
\end{theorem}

(1)\textit{\ }$X_{i}$\textit{\ is a non-empty compact convex set in a
locally convex space }$E_{i}$\textit{;}

(2)\textit{\ }cl$B_{i}$\ \textit{is upper semicontinuous}, \textit{has} 
\textit{non-empty convex values and for each }$x\in X$,\textit{\ }$%
A_{i}(x)\subset B_{i}(x)$\textit{;}

(3) \textit{the set }$W_{i}:=\left\{ x\in X\text{/}\left( A_{i}\cap
P_{i}\right) (x)\neq \emptyset \right\} $\textit{\ is non-empty, open and }$%
K_{i}=$cl$W_{i}$ \textit{\ is} \textit{a} $(n_{i}-1)$\textit{-dimensional
simplex in }$\mathit{X}$\textit{;}

(4)\textit{\ For each convex neighbourhood }$V$\textit{\ of} $0$\textit{\ in}
$X_{i}$, $(A_{i}+V)\cap P_{i}:K_{i}\rightarrow 2^{X_{i}}$ \textit{is a
weakly naturally quasi-concave correspondence}$;$

(5)\textit{\ for each }$x\in K_{i},$\textit{\ }$x_{i}\notin P_{i}(x)$\textit{%
.}

\textit{Then there exists an equilibrium point }$\overline{x}\in X$ \textit{%
\ for }$\Gamma $\textit{,}$\ i.e.$\textit{, for each }$i\in I$\textit{, }$%
\overline{x}_{i}\in \overline{B}_{i}(\overline{x})$\textit{\ and }$A_{i}(%
\overline{x})\cap P_{i}(\overline{x})=\emptyset $\textit{.\medskip }

\begin{proof}
For each\textit{\ }$i\in I$, let \ss $_{i}$ denote the family of all open
convex neighborhoods of zero in $E_{i}.$ Let $V=(V_{i})_{i\in I}\in
\tprod\limits_{i\in I}$\ss $_{i}.$ Since $(A_{i}+V_{i})\cap P_{i}$ is a
weakly concave like correspondence on $K_{i},$ then, from the selection
theorem (Theorem 6), there exists a continuous function $f_{i}^{V_{i}}:K_{i}%
\rightarrow X_{i}$ such that for each $x\in K_{i}$,

$f_{i}^{V_{i}}(x)\in (A_{i}(x)+V_{i})\cap P_{i}(x)\subset
(A_{i}(x)+V_{i})\cap X_{i}.$

It follows that $f_{i}^{V_{i}}(x)\in $cl$(B_{i}(x)+V_{i})$ for $x\in K_{i}.$
Since $X_{i}$ is compact, we have that cl$B_{i}(x)$ is compact for every $%
x\in X$ and cl$(B_{i}(x)+V_{i})=$cl$(B_{i}(x))+$cl$V_{i}$ for every $%
V_{i}\subset E_{i}.$

Define the correspondence $T_{i}^{V_{i}}:X\rightarrow 2^{X_{i}}$, by

$T_{i}^{V_{i}}(x):=\left\{ 
\begin{array}{c}
\{f_{i}^{V_{i}}(x)\}\text{, \ \ \ \ \ \ \ \ \ \ \ \ \ \ \ \ \ \ \ \ \ \ \ \
\ \ \ \ \ \ \ \ \ if }x\in \text{int}_{X}K=W_{i}\text{, } \\ 
\text{cl}(B_{i}(x)+V_{i})\cap X_{i}\text{, \ \ \ \ \ \ \ \ \ \ \ \ \ \ \ \ \
\ if }x\in X\smallsetminus \text{int}_{X}K_{i}\text{;}%
\end{array}%
\right. $

The correspondence $B_{V_{i}}:X\rightarrow 2^{X_{i}}$, defined by $%
B_{V_{i}}(x):=$cl$(B_{i}(x)+V_{i})\cap X_{i}$ is u.s.c. by Lemma 2. Then
following the same line as in Theorem 10, we can prove that $T_{i}^{V_{i}}$
is upper semicontinuous on $X$ and has closed convex values.

Define $T^{V}:X\rightarrow 2^{X}$ by $T^{V}(x):=\underset{i\in I}{\prod }%
T_{i}^{V_{i}}(x)$ for each $x\in X$.

$T^{V}$ is an upper semicontinuous correspondence and it also has non-empty
convex closed values.

Since $X$ is a compact convex set, by Fan's fixed-point theorem [7], there
exists $\overline{x}_{V}\in X$ such that $\overline{x}_{V}\in T^{V}(%
\overline{x}_{V})$, i.e., for each $i\in I$, $(\overline{x}_{V})_{i}\in
T_{i}^{V_{i}}(\overline{x}_{V})$.

We state that $\overline{x}_{V}\in X\setminus \underset{i\in I}{\tbigcup }$%
int$_{X}K_{i}.$

If $\overline{x}_{V}\in $int$_{X}K_{i},$ $(\overline{x}_{V})_{i}\in
T_{i}^{V_{i}}(\overline{x}_{V})=f_{i}(\overline{x}_{V})\in ((A_{i}(\overline{%
x}_{V})+V_{i})\cap P_{i})(\overline{x}_{V})\subset P_{i}(\overline{x}_{V})$,
which contradicts assumption (5).

Hence $(\overline{x}_{V})_{i}\in $cl$(B_{i}(\overline{x}_{V})+V_{i})\cap
X_{i}$ and $(A_{i}\cap P_{i})(\overline{x}_{V})=\emptyset ,$ i.e. $\overline{%
x}_{V}\in Q_{V}$ where

$Q_{V}=\cap _{i\in I}\{x\in X:$ $x_{i}\in $cl$(B_{i}(x)+V_{i})\cap X_{i}$
and $(A_{i}\cap P_{i})(x)=\emptyset \}.$

Since $W_{i}$ is open, $Q_{V}$ is the intersection of non-empty closed sets,
then it is non-empty, closed in $X$.

We prove that the family $\{Q_{V}:V\in \underset{i\in I}{\tprod }\text{\ss }%
_{i}\}$ has the finite intersection property.

Let $\{V^{(1)},V^{(2)},...,V^{(n)}\}$ be any finite set of $\underset{i\in I}%
{\tprod }\text{\ss }_{i}$ and let $V^{(k)}=(V_{i}^{(k)})_{i\in I}$, $%
k=1,...,n.$ For each $i\in I$, let $V_{i}=\underset{k=1}{\overset{n}{\cap }}%
V_{i}^{(k)}$, then $V_{i}\in \text{\ss }_{i};$ thus $V=(V_{i})_{i\in I}\in 
\underset{i\in I}{\tprod }\text{\ss }_{i}$. Clearly $Q_{V}\subset \underset{%
k=1}{\overset{n}{\cap }}Q_{V^{(k)}}$ so that $\underset{k=1}{\overset{n}{%
\cap }}Q_{V^{(k)}}\neq \emptyset .$

Since X is compact and the family $\{Q_{V}:V\in \underset{i\in I}{\tprod }%
\text{\ss }_{i}\}$ has the finite intersection property, we have that $\cap
\{Q_{V}:V\in \underset{i\in I}{\tprod }\text{\ss }_{i}\}\neq \emptyset .$
Take any $\overline{x}\in \cap \{Q_{V}:V\in \underset{i\in I}{\tprod \text{%
\ss }_{i}}\},$ then for each $i\in I$ and each $V_{i}\in \text{\ss }_{i},$ $%
\overline{x}_{i}\in $cl$(B_{i}(\overline{x})+V_{i})\cap X_{i}$ and $%
(A_{i}\cap P_{i})(\overline{x})=\emptyset ;$ but then $\overline{x_{i}}\in $%
cl$(B_{i}(\overline{x}))$ by Lemma\emph{\ 1 }and $(A_{i}\cap P_{i})(%
\overline{x})=\emptyset $ for each $i\in I$ \ so that $\overline{x}$ is an
equilibrium point of $\Gamma $ in X. \medskip
\end{proof}

The last two theorems can be compared with Zheng's theorems 3.1 and 3.2 in
[24] and Zhou's theorems 5 and 6 in [25] where the constraint
correspondences have continuous selections on a closed subset $C_{i}\subset
X $\ \ which contains the set\ $\left\{ x\in X:\left( A_{i}\cap P_{i}\right)
\left( x\right) \neq \emptyset \right\} .$

To find the equilibrium point in Theorem 14, we use Wu's fixed point theorem
for correspondences cl$B_{i}$ which are lower semicontinuous and we need a
non-empty compact metrizable set $D_{i}$ in $X_{i}$ for each $i\in I.$ The
spaces $X_{i}$ are not compact.\medskip

\begin{theorem}
\textit{Let }$\Gamma =(X_{i},A_{i},P_{i},B_{i})_{i\in I}$\textit{\ \ be an
abstract economy, where }$I$\textit{\ is a (possibly uncountable) set of
agents such that for each }$i\in I:$
\end{theorem}

(1)\textit{\ }$X_{i}$\textit{\ is a non-empty convex set in a Hausdorff
locally convex space }$E_{i}$ \textit{and there exists a nonempty compact
metrizable subset }$D_{i}$\textit{\ of }$X_{i}$\textit{\ containing all
values of the correspondences }$A_{i},P_{i}$\textit{\ and }$B_{i}$\textit{;}

(2)\textit{\ }cl$B_{i}$\textit{\ is lower semicontinuous with non-empty
convex values;}

(3) \textit{there exists a }$(n_{i}-1)$\textit{-dimensional simplex }$K_{i}$%
\textit{\ in }$X$\textit{\ and }

$\ \ \ \ \ \ W_{i}:$\textit{\ }$=\left\{ x\in X\text{ / }\left( A_{i}\cap
P_{i}\right) (x)\neq \emptyset \right\} \subset $int$_{X}(K_{i})$\textit{;}

(4)\textit{\ }cl$B_{i}$\textit{\ has \ the (WNQS)-property on }$K_{i}$%
\textit{;}

\textit{Then there exists an equilibrium point }$\overline{x}\in D$ \textit{%
\ for }$\Gamma $\textit{,}$\ i.e.$\textit{, for each }$i\in I$\textit{, }$%
\overline{x}_{i}\in $cl$B_{i}(\overline{x})$\textit{\ and }$A_{i}(\overline{x%
})\cap P_{i}(\overline{x})=\emptyset $\textit{.\medskip }

\begin{proof}
\textit{Proof.} Since cl$B_{i}$ has the WNQS property on $K_{i}$, it follows
that there exists a weakly concave like correspondence $F_{i}:X\rightarrow
2^{D_{i}}$ such that $F_{i}(x)\subset $cl$B_{i}(x)$ and $x_{i}\notin
F_{i}(x) $ for each $x\in K_{i}$.

$K_{i}$ is \emph{a }$(n_{i}-1)$\emph{- }dimensional simplex\emph{, }then,
from the selection theorem, there exists \emph{\ }a continuous function $%
f_{i}:K_{i}\rightarrow D_{i}$ such that $f_{i}(x)\in F_{i}(x)$ for each $%
x\in K_{i}.$ Because $x_{i}\notin F_{i}(x)$ for each $x\in K_{i},$ we have
that $x_{i}\neq f_{i}(x)$ for each $x\in K_{i}.$

Define the correspondence $T_{i}:X\rightarrow 2^{D_{i}}$, by $%
T_{i}(x):=\left\{ 
\begin{array}{c}
\{f_{i}(x)\}\text{, if }x\in K_{i}\text{, } \\ 
\text{cl}B_{i}(x)\text{, if }x\notin K_{i}\text{.}%
\end{array}%
\right. $

$T_{i}$ is lower semicontinuous on $X$ and has closed convex values.

Let $U$ be a closed subset of $\ X_{i}$, then

$U^{^{\prime }}:=\{x\in X$ $\mid $ $T_{i}(x)\subset U\}$\ \ \ =$\{x\in K_{i}$
$\mid $ $T_{i}(x)\subset U\}\cup \{x\in X\setminus K_{i}$ $\mid $ $%
T_{i}(x)\subset U\}$

\ \ \ =$\left\{ x\in K_{i}\text{ }\mid \text{ }f_{i}(x,y)\in U\right\} \cup
\left\{ x\in X\mid \text{ cl}B_{i}(x)\subset U\right\} $

\ \ \ =$((f_{i})^{-1}(U)\cap K_{i})\cup $.$\left\{ x\in X\mid \text{ cl}%
B_{i}(x)\subset U\right\} .$

$U^{^{\prime }}$ is a closed set, because $K_{i}$ is closed, $f_{i}$ is a
continuous function on $K_{i}$ and the set $\left\{ x\in X\mid \text{ cl}%
B_{i}(x)\subset U\right\} $ is closed since cl$B_{i}(x)$ is l.s.c. Then $%
T_{i}$ is lower semicontinuous on $X$ and has non-empty closed convex values.

By Theorem 3 (Wu's fixed-point theorem) applied for the correspondences $%
S_{i}=T_{i}$ and $T_{i}:X\rightarrow 2^{D_{i}},$ there exists $\overline{x}%
\in D$ such that for each $i\in I$, $\overline{x}_{i}\in T_{i}(\overline{x})$%
. If $\overline{x}\in W_{i}$ for some $i\in I$, then $\overline{x}_{i}=f_{i}(%
\overline{x})$, which is a contradiction.

Therefore, $\overline{x}\notin W_{i}$, and hence $(A_{i}\cap P_{i})(%
\overline{x})=\emptyset $. Also, for each $i\in I$, we have $\overline{x}%
_{i}\in T_{i}(\overline{x})$, and then $\overline{x}_{i}\in $cl$B_{i}(%
\overline{x}).\medskip $
\end{proof}

In Theorem 15 the sets $X_{i}$ are non-empty compact convex in locally
convex spaces $E_{i}.$ As in Theorem 13, we first obtain equilibria for $%
\Gamma _{V},$ and then, the proof coincides with the proof of Theorem
13.\medskip

\begin{theorem}
\textit{Let }$\Gamma =(X_{i},A_{i},P_{i},B_{i})_{i\in I}$\textit{\ \ be an
abstract economy, where }$I$\textit{\ is a (possibly uncountable) set of
agents such that for each }$i\in I:$
\end{theorem}

(1)\textit{\ }$X_{i}$\textit{\ is a non-empty compact convex set in a
locally convex space }$E_{i}$\textit{;}

(2)\textit{\ }cl$B_{i}$\textit{\ is upper semicontinuous with non-empty
convex values;}

(3)\textit{\ the set }$W_{i}:$\textit{\ }$=\left\{ x\in X\text{ / }\left(
A_{i}\cap P_{i}\right) (x)\neq \emptyset \right\} $ \textit{is open and
there exists a }

$(n_{i}-1)$\textit{-dimensional simplex }$K_{i\text{ }}$\textit{in }$X$%
\textit{\ such that} $W_{i}\subset $int$_{X}(K_{i})$.

(3)\textit{\ }cl$B_{i}$\textit{\ has \ the (e-WNQS)-property on }$K_{i}$%
\textit{.}

\textit{Then there exists an equilibrium point }$\overline{x}\in X$ \textit{%
\ for }$\Gamma $\textit{,}$\ i.e.$\textit{, for each }$i\in I$\textit{, }$%
\overline{x}_{i}\in \overline{B}_{i}(\overline{x})$\textit{\ and }$A_{i}(%
\overline{x})\cap P_{i}(\overline{x})=\emptyset $\textit{.\medskip }

\begin{proof}
For each\textit{\ }$i\in I$, let \ss $_{i}$ denote the family of all open
convex neighborhoods of zero in $E_{i}.$Let $V=(V_{i})_{i\in I}\in
\tprod\limits_{i\in I}$\ss $_{i}.$ Since cl$B_{i}$ has the e-WNQS property
on $K_{i}$, it follows that there exists a weakly concave like
correspondence $F_{i}^{V_{i}}:X\rightarrow 2^{X_{i}}$ such that $%
F_{i}^{V_{i}}(x)\subset clB_{i}(x)+V_{i}$ and $x_{i}\notin F_{i}^{V_{i}}(x)$
for each $x\in K_{i}$.

$K_{i}$ is a\emph{\ }$(n_{i}-1)$\emph{- }dimensional simplex,\emph{\ }then,
from the selection theorem, there exists \emph{\ }a continuous function $%
f_{i}^{V_{i}}:K_{i}\rightarrow X_{i}$ such that $f_{i}^{V_{i}}(x)\in
F_{i}^{V_{i}}(x)$ for each $x\in K_{i}.$ Because $x_{i}\notin
F_{i}^{V_{i}}(x)$ for each $x\in K_{i},$ we have that $x_{i}\neq
f_{i}^{V_{i}}(x)$ for each $x\in K_{i}.$

Define the correspondence $T_{i}^{V_{i}}:X\rightarrow 2^{X_{i}}$, by

$T_{i}^{V_{i}}(x):=\left\{ 
\begin{array}{c}
\{f_{i}^{V_{i}}(x)\}\text{, \ \ \ \ \ \ \ \ \ \ \ \ \ \ \ \ \ \ \ \ \ \ \ \
if }x\in \text{int}_{X}K_{i}\text{, } \\ 
\text{cl}(B_{i}(x)+V_{i})\cap X_{i}\text{, \ \ \ \ \ \ \ \ \ if }x\in
X\smallsetminus \text{int}_{X}K_{i}\text{;}%
\end{array}%
\right. $

$B_{V_{i}}:X\rightarrow 2^{X_{i}},$ $B_{V_{i}}(x)=$cl$(B_{i}(x)+V_{i})\cap
X_{i}=($cl$B_{i}(x)+$cl$V_{i})\cap X_{i}$ is upper semicontinuous by Lemma 2%
\emph{.}

Let $U$ be an open subset of $\ X_{i}$, then

$U^{^{\prime }}:=\{x\in X$ $\mid $ $T_{i}^{V_{i}}(x)\subset U\}$

\ \ \ =$\{x\in $int$_{X}K_{i}$ $\mid $ $T_{i}^{V_{i}}(x)\subset U\}\cup
\{x\in X\setminus $int$_{X}K_{i}$ $\mid $ $T_{i}^{V_{i}}(x)\subset U\}$

\ \ \ =$\left\{ x\in \text{int}_{X}K_{i}\text{ }\mid \text{ }%
f_{i}^{V_{i}}(x,y)\in U\right\} \cup \left\{ x\in X\mid \text{ }(\text{cl}%
B_{i}(x)+\overline{V_{i}})\cap X_{i}\subset U\right\} $

\ \ \ =$((f_{i}^{V_{i}})^{-1}(U)\cap $int$_{K}K_{i})\cup $.$\left\{ x\in
X\mid \text{ }(\text{cl}B_{i}(x)+\overline{V_{i}})\cap X_{i}\subset
U\right\} .$

$U^{^{\prime }}$ is an open set, because int$_{X}K_{i}$ is open, $%
f_{i}^{V_{i}}$ is a continuous function on $K_{i}$ and the set $\left\{ x\in
X\mid \text{ }(\text{cl}B_{i}(x)+\text{cl}V_{i})\cap X_{i}\subset U\right\} $
is open since $($cl$B_{i}(x)+$cl$V_{i})\cap X_{i}$ is u.s.c. Then $%
T_{i}^{V_{i}}$ is upper semicontinuous on $X$ and has closed convex values.

Define $T^{V}:X\rightarrow 2^{X}$ by $T^{V}(x):=\underset{i\in I}{\prod }%
T_{i}^{V_{i}}(x)$ for each $x\in X$.

$T^{V}$ is an upper semicontinuous correspondence and it has also non-empty
convex closed values.

Since $X$ is a compact convex set, by Fan's fixed-point theorem [7], there
exists $\overline{x}_{V}\in X$ such that $\overline{x}_{V}\in T^{V}(%
\overline{x}_{V})$, i.e., for each $i\in I$, $(\overline{x}_{V})_{i}\in
T_{i}^{V_{i}}(\overline{x}_{V})$. If $\overline{x}_{V}\in $int$_{X}K_{i},$ $(%
\overline{x}_{V})_{i}=f_{i}^{V_{i}}(\overline{x}_{V})$, which is a
contradiction.

Hence $(\overline{x}_{V})_{i}\in $cl$(B_{i}(\overline{x}_{V})+V_{i})\cap
X_{i}$ and $(A_{i}\cap P_{i})(\overline{x}_{V})=\emptyset ,$ i.e. $\overline{%
x}_{V}\in Q_{V}$ where

$Q_{V}=\cap _{i\in I}\{x\in X:$ $x_{i}\in $cl$(B_{i}(x)+V_{i})\cap X_{i}$
and $(A_{i}\cap P_{i})(x)=\emptyset \}.$

Since $W_{i}$ is open, $Q_{V}$ is the intersection of non-empty closed sets,
then it is non-empty, closed in X.

We prove that the family $\{Q_{V}:V\in \underset{i\in I}{\tprod }\text{\ss }%
_{i}\}$ has the finite intersection property.

Let $\{V^{(1)},V^{(2)},...,V^{(n)}\}$ be any finite set of $\underset{i\in I}%
{\tprod \text{\ss }_{i}}$ and let $V^{(k)}=(V_{i}^{(k)})_{i\in I}$, $%
k=1,...,n.$ For each $i\in I$, let $V_{i}=\underset{k=1}{\overset{n}{\cap }}%
V_{i}^{(k)}$, then $V_{i}\in \text{\ss }_{i};$ thus $V=(V_{i})_{i\in I}\in 
\underset{i\in I}{\tprod }\text{\ss }_{i}.$ Clearly $Q_{V}\subset \underset{%
k=1}{\overset{n}{\cap }}Q_{V^{(k)}}$ so that $\underset{k=1}{\overset{n}{%
\cap }}Q_{V^{(k)}}\neq \emptyset .$Since $X$ is compact and the family $%
\{Q_{V}:V\in \underset{i\in I}{\tprod }\text{\ss }_{i}\}$ has the finite
intersection property, we have that $\cap \{Q_{V}:V\in \underset{i\in I}{%
\tprod }\text{\ss }_{i}\}\neq \emptyset .$ Take any $\overline{x}\in \cap
\{Q_{V}:V\in \underset{i\in I}{\tprod }\text{\ss }_{i}\},$ then for each $%
i\in I$ and each $V_{i}\in \text{\ss }_{i},$ $\overline{x}_{i}\in $cl$(B_{i}(%
\overline{x})+V_{i})\cap X_{i}$ and $(A_{i}\cap P_{i})(\overline{x}%
)=\emptyset ;$ but then $\overline{x}_{i}\in $cl$(B_{i}(\overline{x}))$ from
Lemma\emph{\ 1 }and $(A_{i}\cap P_{i})(\overline{x})=\emptyset $ for each $%
i\in I$ \ so that $\overline{x}$ is an equilibrium point of $\Gamma $ in X.
\medskip
\end{proof}

This work was supported by the strategic grant POSDRU/89/1.5/S/58852,
Project "Postdoctoral programme for training scientific researchers"
cofinanced by the European Social Found within the Sectorial Operational
Program Human Resources Development 2007-2013.

\begin{center}
\bigskip
\end{center}

\bigskip

\end{document}